\documentclass{amsart}
\usepackage{amssymb}
\usepackage{graphicx}
\newtheorem{thm}{Theorem}

\newtheorem{prop}[thm]{Proposition}

\def\({\left(}
\def\){\right)}

\newtheorem{lema}{Lemma}[section]

\newtheorem*{teorema*}{Theorem}

\newtheorem{remark}[lema]{Remark}

\newtheorem{theorem}[lema]{Theorem}

\newtheorem{definition}[lema]{Definition}

  {\par \hfill \fbox{}}
  {\par \hfill \fbox{}}
  {\par \hfill }
  {\par \hfill }

\def\beq{\begin{equation}}
\def\eeq{\end{equation}}

\def\epsilon{\varepsilon}

\begin{document}

\title{Fixed point theorems in cone metric spaces via $w-$distance over topological module}

\author{Shallu Sharma}
\address{Department of Mathematics, University of Jammu } 
\email{shallujamwal09@gmail.com}

\author{Iqbal Kour}
\address{Department of Mathematics, University of Jammu} \email{iqbalkour208@gmail.com}
\author{Pooja Saproo}
\address{Department of Mathematics,University Of Jammu, India}
\email{poojasaproo1994@gmail.com}


\date{}

\begin{abstract}
The concept of cone metric spaces with $w-$distance was introduced by H. Lakzian and F. Arabyani \cite{Lak} in $2009.$ In $2020,$ Branga and Olaru \cite{Branga} put forth the idea of cone metric spaces over topological module. In this paper, we compose these concepts together and introduce cone metric spaces with $w-$distance over topological module and then prove some fixed point theorems.   
\end{abstract}

\maketitle{ }

\section{Introduction}
The work of Huang and Zhang \cite{Huang} introduced the idea of cone metric spaces and established the Banach Contraction principle in this context. For more results on cone metric spaces, one can see \cite{Ciric, Dor, Jan, Kad 1, Kad 2, Kad 3}. Additionally, a number of authors have investigated fixed point theorems in cone metric spaces (see, for instance \cite{Abbas 1, Abbas 2, Azam, Ill1, Ill2}). 
Metric Spaces with $w$-distance was introduced by Kada et. al \cite{Kada} in 1996. H. Lakzian and F. Arabyani \cite{Lak} introduced cone metric space with $w$-distance and proved some fixed point theorems. In 2020, cone metric spaces over topological module was introduced by Branga and Olaru \cite{Branga}. In this paper, we extend the concept of $w$-distance to cone metric spaces over topological module and introduce a new notion ``cone metric spaces with $w$-distance over topological module" and prove some fixed point theorems.

\section{Preliminaries}     
\begin{definition}\cite{AMWG}
Let $(G, +)$ be a group and $\leq$ be a partial order relation on $G$. Then $G$ is said to be a {\bf partially ordered group} if translation in $G$ preserves order: \\ $$ x \leq y \Rightarrow c+x+d \leq c+y+d, ~\forall~ x,y,c,d \in G$$
\end{definition}

\begin{definition} \cite{Stein}
Consider a ring $(R, +, .)$ with identity $1$, where $1 \neq 0$ and $\leq $ is a partial order on $R$. Then $R$ is said to be a {\bf partially ordered ring} if: 
\begin{enumerate}
\item $(R, +, .)$ is a partially ordered group;
\item $c \geq 0$ and $d \geq 0$ implies $c.d \geq 0$ for all $c,d \in R$.
\end{enumerate}
The positive cone of $R$ is denoted by $R^{+} =\{ r \in R: r \geq 0\}$.  
\end{definition}

\begin{remark}
$\mathcal{U}(R)$ denotes the set of invertible elements of $R$ and $\mathcal{U}_{+}(R)$ denotes $\mathcal{U}(R) \cap R^{+}$.
\end{remark}

\begin{definition} \cite{Warner}
Consider an abelian group $(G,+)$. $G$ is said to be a {\bf topological group} if $G$ is endowed with a topology $\mathcal{G}$ and satisfies the following conditions:
\begin{enumerate}
\item For $(g_{1}, g_{2}) \in G\times G$, the map $(g_{1}, g_{2}) \rightarrow  g_{1}+g_{2}$ is continuous, where $g_{1}+g_{2}$ is an element of $G$ and $G \times G$ is endowed with the product topology;
\item For $g \in G$, the map $g \rightarrow -g$ is continuous, where $-g$ is an element of $G$.
\end{enumerate}
 We shall denote topological group as $(G,+, \mathcal{G})$ or $(G, \mathcal{G})$.
\end{definition}

\begin{definition} \cite{Warner}
A ring $(R, +, .)$ is said to be a {\bf topological ring} if $R$ considered with the topology $\mathcal{R}$ such that $(R, +, \mathcal{R})$ is a topological group and if for $(r_{1},r_{2}) \in R \times R$, the map $(r_{1}, r_{2}) \rightarrow r_{1}.r_{2}$ is continuous, where $r_{1}.r_{2}$ is an element of $R$ and $R \times R$ is endowed with the product topology.
\end{definition}
 
$(R, +, ., \mathcal{R})$ is said to be a {\bf Hausdorff topological ring} \cite{Warner} if the topology $\mathcal{R}$ is Hausdorff. We shall denote the topological ring as $(R, +,., \mathcal{R})$ or $(R, \mathcal{R})$.

\begin{definition} \cite{Warner}
Consider a topological ring $(R, \mathcal{R})$. A left $R-$module $(E,+,.)$ is said to be a {\bf topological $R-$module} if we define a topology $\mathcal E$ on $E$ in such a way that $(E,+)$ is a topological abelian group and satisfies the following condition:\\
For $(r,x)\in R\times E,$ the map $(r,x) \rightarrow r.x$ is continuous, where $r.x$ is an element of $E.$ We shall denote a topological left $R-$module as $(E,+,.,\mathcal E)$ or $(E,\mathcal E).$  
\end{definition}

For subsequent definitions and results in this section, we refer \cite{Branga}.

\begin{definition} 
Consider a topological module $(E,+,.,\mathcal E).$ A set $P\subset E$ is said to be a {\bf cone} if :
\begin{itemize}
\item [(1)] $P\neq \phi,\bar{P}=P,P\neq \{0_{E}\};$
\item [(2)] $c.x+d.y\in P,$ whenever $x,y\in P$ and $c,d\in R^{+};$
\item [(3)] $x\in P$ and $-x\in P$ implies $x=0_{E}.$
\end{itemize}
\end{definition}

\begin{remark}
Throughout this paper $P^{\circ}$ shall denote interior of $P$ and $\bar{P}$ shall denote closure of $P$. The cone $P$ is said to be {\bf solid} if $P^{\circ}\neq \phi.$ Let $P$ be a cone contained in $E.$ A {\bf partial ordering} $\leq_{P}$ with respect to $P$ is defined by $y-x\in P$ if and only if $x\leq_{P}y.$ Further, $x<_{P}y$ denotes that $x\leq_{P} y,x\neq y.$ Also, the partial ordering $x\ll y$ indicates that $y-x\in P^{\circ}.$\\ 
\end{remark}  

Consider the following hypotheses:
\begin{itemize} 
\item [(Hyp I)] Let $(R,+,\odot,\mathcal R)$ be a Hausdorff topological ring so that:
\item [(1)] $\mathcal{U}_{+}(R)$ is non-empty.
\item [(2)] $0_{R}$ is a limit point of $\mathcal{U_{+}}(R).$
\item [(3)] $\leq_{R}$ denotes the partial ordering on $R.$
\item [(Hyp II)] $(E,+,.,\mathcal E)$ is considered as a topological left $R-$module.
\item [(Hyp III)] $P$ is a solid cone in $E.$
\end{itemize}

\begin{prop}\label{proposition:1} 
Let $(E,+,.,\mathcal E)$ be a topological left $R-$module and $P$ be a cone in $E$ such that the hypotheses $Hyp I,Hyp II$ and $Hyp III$ are satisfied. Then 
\begin{itemize}
\item [(1)] $P^{\circ}+P^{\circ}\subseteq P^{\circ}.$
\item [(2)] $\alpha \odot P^{\circ}\subseteq P^{\circ},$ where $\alpha$ is a element of $\mathcal U_{+}(R).$
\item [(3)] If $u\leq_{P}v$ and $\beta \in R^{+},$ then $\beta \odot u\leq_{P}\beta \odot v.$
\item [(4)] If $x\leq_{P}y$ and $y\ll z,$ then $x\ll z.$
\item [(5)] If $x\ll y$ and $y\leq_{P}z,$ then $x\ll z.$
\item [(6)] If $x\ll y$ and $y\ll w,$ then $x\ll w.$
\item [(7)] If $0_{E}\leq_{P} x\ll u,~\forall~u\in P^{\circ},$ then $x=0.$
\item [(8)] If $0_{E}\ll u$ and $\{b_{n}\}$ is a sequence in $E$ so that $b_{n}\to 0_{E},$ then there is  a natural number $m_{0}$ so that $b_{n}\ll u$ for $n \geq m_{0}.$
\end{itemize}
\end{prop} 
 
From now onwards, $(R,+,.,\mathcal{R})$ denotes a Hausdorff topological ring.  

\begin{definition} 
A directed set is defined to be a partially ordered set $(\Lambda,\leq)$ such that the following condition is satisfied:\\
For $\lambda_{1}, \lambda_{2}$ in $\Lambda$ there exists $\lambda_{3}$ in $\Lambda$ so that $\lambda_{1}\leq \lambda_{3}$ and $\lambda_{2}\leq \lambda_{3}.$
\end{definition}

\begin{definition}
A sequence $\{x_{\lambda}\}_{\lambda\in \Lambda}$ contained in $R$ is a family of elements indexed by a directed set.
\end{definition}

\begin{definition}
A family $\{x_{\lambda}\}_{\lambda \in \Lambda}$ contained in $R$ is said to be convergent to an element $x\in R$ if for each neighborhood $U$ of $x$ there exists $\lambda_{0}\in \Lambda$ such that $x_{\lambda}$ is an element of $U$ for every $\lambda \in \Lambda,$ where $\lambda \geq \lambda_{0}.$ 
\end{definition}

\begin{definition}
A sequence $\{x_{\lambda}\}_{\lambda\in \Lambda}$ is said to be a Cauchy sequence if for each neighborhood $U$ of $0_{R}$ there exists $\lambda_{0}\in \Lambda$ so that $x_{\lambda_{1}}-x_{\lambda_{2}}\in U$ for each $\lambda_{0}\leq \lambda_{1}$ and $\lambda_{0}\leq \lambda_{2}.$
\end{definition}

\begin{theorem}
Every convergent sequence $\{x_{\lambda}\}_{\lambda\in \Lambda}$ contained in $R$ is a Cauchy sequence.
\end{theorem}
To define the summability of the family of elements of a topological ring, let $\mathcal H(\Lambda)$ be the set of all finite sets contained in $\Lambda$ directed by the inclusion $\subseteq.$

\begin{definition}
An element $t$ of $R$ is sum of family $\{x_{\lambda}\}_{\lambda\in \Lambda}$ contained in $R$ if the sequence $\{t_{I}\}_{I\in \mathcal{H}(\Lambda)}$ is convergent to $t,$ where for each $I\in \mathcal{H}(\Lambda),$ $$t_{I}=\sum_{\lambda\in I}x_{\lambda}.$$
\end{definition}

The family $\{t_{\lambda}\}_{\lambda\in I}$ is said to be summable if $\{t_{\lambda}\}_{\lambda\in I}$ has a sum $t$ in $R.$

\begin{definition}
A family $\{x_{\lambda}\}_{\lambda\in I}$ contained in $R$ is said to satisfy Cauchy condition if for each neighborhood $U$ of $0_{R}$ there exists $I_{U}$ in $\mathcal{H}(\Lambda)$ so that {\bf $\sum_{\lambda\in J}x_{\lambda}\in U,$} for each $J\in \mathcal{H}(\Lambda)$ disjoint with $I_{U}.$
\end{definition}

\begin{theorem}
A family $\{x_{\lambda}\}_{\lambda\in \Lambda}$ contained in $R$ satisfies Cauchy condition if and only if $\{t_{I}\}_{I \in \mathcal{H}(\Lambda)}$ is a Cauchy sequence.
\end{theorem}

\begin{theorem}\label{theorem:1}
Let $\{x_{\lambda}\}_{\lambda\in \Lambda}$ be a summable family in $R.$ Then for each neighborhood $U$ of $0_{R},$ there exists $J$ in $\mathcal{H}(\Lambda)$ so that $x_{\lambda}\in U$ for each $\lambda\in \Lambda\setminus J.$  
\end{theorem}

\begin{definition}
Let $(R,+,.,\mathcal R)$ be a topological ring. Then $R$ is said to be complete if the topological additive group of ring $(R,+,\mathcal R)$ is complete. 
\end{definition}

\begin{definition} 
Let $X\neq \phi$ be a set, $(E,+,.,\mathcal E)$ be a topological left $R-$module and the map $d:X\times X\to E$ satisfies:
\begin{itemize}
\item [(1)] $d(x,y)\geq_{P} 0_{E}~\forall~x,y\in X.$
\item [(2)] $d(x,y)=0_{E}$ if and only if $x=y~\forall~x,y\in X.$
\item [(3)]$d(x,y)=d(y,x)~\forall~x,y\in X.$
\item [(4)] $d(x,y)\leq_{P} d(x,z)+d(z,y)~\forall~x,y\in X.$
\end{itemize}
Then $d$ is said to be a cone metric on $X$ and the pair $(X,d)$ is said to be a cone metric space over topological left $R-$module $E.$ 
\end{definition}

\begin{definition}
Let $(X,d)$ be a cone metric space over the topological left $R-$module E, $x$ be an element of $X$ and $\{x_{n}\}$ be a sequence in $X.$ Then
\begin{itemize}
\item [(1)] $\{x_{n}\}$ is said to be convergent to $x$ if for each $u\gg 0,$ there is a natural number $N$ such that $d(x_{n},x)\ll u,~\forall~n>N.$
\item [(2)] $\{x_{n}\}$ contained in $X$ is said to be a Cauchy sequence if for each $u\gg 0,$ there is a natural natural number $N$ so that $d(x_{n},x_{m})\ll u,~\forall~m,n>N.$   
\end{itemize}
\end{definition}

If every Cauchy sequence in the cone metric space $(X,d)$ is convergent then $(X,d)$ is said to be complete.

\section{Cone metric spaces over topological module}

In this section, some new results in cone metric spaces over topological module are put forth.

\begin{prop}
Let $(E,+,.,\mathcal E)$ be a topological left $R$-module and $P$ be a cone in $E$ so that the hypotheses $HypI,HypII$ and $HypIII$ are satisfied. If
\begin{itemize}
\item [(i)] $x\leq_{P}y,y\leq_{P}z,$ then $x\leq_{P}z.$
\item [(ii)]$x\ll y$ and $\lambda\in R^{+},$ then $\lambda \odot x\ll \lambda \odot y.$ 
\item [(iii)] $x\ll y$ and $\lambda\in R^{+},$ then $\lambda \odot x\leq_{P} \lambda \odot y.$  
\end{itemize}
\end{prop}

\begin{proof}
\begin{itemize}
\item [(i)] If $x\leq_{P} y,$ then $y-x\in P.$ Also, if $y\leq_{P} z,$ then $y-z\in P.$ Now, $z-x=(z-y)+(y-x)\in P+P\subset P.$ Hence $z-x\in P.$ Therefore, $x\leq_{P} z.$
\item [(ii)] If $x\ll y,$ then $y-x\in P^{\circ}.$ Also, for $\lambda\in {R}^{+},\lambda \odot(y-x)\in P^{\circ}.$ So, $\lambda \odot x\ll \lambda \odot y.$
\item [(iii)] Proof follows from $(ii).$
\end{itemize}
\end{proof}

\begin{theorem}\label{theorem:3}
Let $(X,d)$ be a cone metric space over a topological left $R$-module E. Then every convergent sequence is a Cauchy sequence.
\end{theorem}

\begin{proof}
Suppose that $\{x_{n}\}$ is a sequence in $X$ convergent to $x.$ Then for all $u\gg 0$ there is a natural number $N$ such that $d(x_{n},x)\ll u/2$ and $d(x_{m},x)\ll u/2,~\forall~n,m\geq N.$ Now,
\begin{eqnarray*}
d(x_{n},x_{m})&\leq_{P}&d(x_{n},x)+d(x,x_{m})\\
              &\ll&\frac{u}{2}+\frac{u}{2}\\
							&=&u.
\end{eqnarray*}
Hence the proof.
\end{proof}

\begin{theorem}\label{theorem:2}
Let $(X,d)$ be a cone metric space over a topological left $R$-module E, $P$ be a solid cone in $E$ so that hypothesis $H_{1},H_{2}$ and $H_{3}$ are satisfied. Consider sequences $\{x_{n}\}$ in $X$ and $\{b_{n}\}\in E$ such that $b_{n}\to 0_{E}.$
\begin{itemize}
\item [(i)] If $d(x_{n},x_{m})\leq b_{n}$ for each $n\in \mathbb{N}$ with $m>n,$ then $\{x_{n}\}$ is a Cauchy sequence.
\item [(ii)] If $d(x_{n},x)\leq b_{n}$ for each $n\in \mathbb{N},$ then $\{x_{n}\}$ is convergent to $x.$
\item [(iii)] If $d(x_{n},x)\leq b_{n}$ for each $n\in \mathbb{N},$ then $\{x_{n}\}$ is a Cauchy sequence.
\end{itemize} 
\end{theorem}

\begin{proof}
  
\begin{itemize}
\item [(i)] Since $b_{n}\to 0_{E}.$ Then by Proposition \ref{proposition:1} $(8),$ we see that for $u\gg 0,$ there exists a natural number $m_{0}$ such that $b_{n}\ll u$ for $n\geq m_{0}.$ Hence $d(x_{n},x_{m})\leq b_{n}$ implies that $d(x_{n},x_{m})\ll u,~\forall~n,m>N.$ This proves that $\{x_{n}\}$ is a Cauchy sequence.
\item [(ii)] Since $b_{n}\to 0_{E}.$ Then by Proposition \ref{proposition:1} $(8),$ we see that for $u\gg 0,$ there exists a natural number $N$ such that $b_{n}\ll u$ for $n\geq N.$ Hence $d(x_{n},x)\leq b_{n}$ implies that $d(x_{n},x)\ll u$ for all $N\in \mathbb{N}$ such that $n>N.$ This proves that $\{x_{n}\}$ is convergent to $x.$
\item [(iii)] Proof follows from Theorem \ref{theorem:3} and $(ii).$   
\end{itemize}

\end{proof}

\begin{remark}\label{remark:1}\cite{Branga}
Let $(X,d)$ be a cone metric space over a topological left $R-$module E and consider a sequence $\{x_{n}\}$ in $X.$ If $\{x_{n}\}$ is convergent to $x$ and $\{x_{n}\}$ is convergent to $y,$ then $x=y.$
\end{remark}

\begin{theorem}
Let $(X,d)$ be a cone metric space over a topological left $R-$module $E,P$ be a solid cone in $E$ such that the hypothesis $H_{1},H_{2}$ and $H_{3}$ are satisfied, $\{x_{n}\}$ and $\{y_{n}\}$ be sequences in $X,\{a_{n}\},\{b_{n}\}$ be two sequences in $E$ convergent to $0_{E}$ and $y,z\in X.$
\begin{itemize}
\item [(i)] If $d(x_{n},y)\leq_{P} a_{n}$ and $d(x_{n},z)\leq_{P}b_{n}$ for any natural number $n,$ then $y=z.$
\item [(ii)] If $d(x_{n},y_{n})\leq_{P} a_{n}$ and $d(x_{n},z)\leq_{P}a_{n},$ then $\{y_{n}\}$ is convergent to $z.$
\item [(iii)] If $d(x_{n},y_{n})\leq_{P} a_{n}$ and $d(x_{n},z)\leq_{P}a_{n},$ then $\{y_{n}\}$ is a Cauchy sequence.
\end{itemize}
 \end{theorem}

\begin{proof}
\begin{itemize}
\item [(i)] Since $d(x_{n},y)\leq_{P} a_{n}$ and $d(x_{n},z)\leq_{P}b_{n},n\in \mathbb{N}.$ Then using Theorem \ref{theorem:2} $\{x_{n}\}$ is convergent to $y$ and $\{x_{n}\}$ is convergent to $z.$ By using Remark \ref{remark:1} we see that $y=z.$
\item [(ii)] Since $\{a_{n}\}$ is convergent to $0_{E}.$ Then by using Proposition \ref{proposition:1} (8), for each $u\gg 0,$ we have $d(x_{n},y_{n})\ll u$ and $d(x_{n},z)\ll u.$ Now, $d(y_{n},z)\leq_{P}d(y_{n},x_{n})+d(x_{n},z)\ll u+u=2u.$ Hence $\{y_{n}\}$ is convergent to $z.$
\item [(iii)] Proof follows from Theorem \ref{theorem:3} and $(ii).$   
\end{itemize}     
\end{proof}

\section{Fixed point theorems via w-distance}

This section deals with the main results of this paper. First, we introduce $w-$distance in cone metric spaces over topological module. Next we discuss some results regarding the same. Further, some fixed point theorems in cone metric spaces via $w-$distance over topological module have been proved.

\begin{definition}
Let $(X,d)$ be a cone metric space over a topological left $R-$module E. Then a map $q:X\times X\to E$ is said to be a $w-$distance on $X$ if it satisfies the following conditions:
\begin{itemize}
\item [(1)] $q(x,y)\geq_{P} 0_{E} ~\forall~x,y\in X.$
\item [(2)] $q(x,z)\leq_{P}q(x,y)+q(y,z) ~\forall~x,y,z\in X.$
\item [(3)] $q(x,.)\to E$ is lower semi-continuous $~\forall~x\in X.$
\item [(4)] For each $u\gg 0,$ there exists $v\gg 0$ so that $q(z,x)\ll v$ and $q(z,y)\ll v$ implies that  $d(x,y)\ll u$ where $u,v\in E.$
\end{itemize}
\end{definition}

\begin{definition}
Let $(X,d)$ be a cone metric space over a topological left $R-$module $E,q$ be a $w-$distance on $X,x\in X$ and a sequence $\{x_{n}\}$ in $X$. Then
\begin{itemize}
\item [(1)] $\{x_{n}\}$ is said to be $q-$Cauchy if for each $u\in E,u\gg 0,$ there exists $\mathbb{N}\in \mathbb{Z}^{+}$ so that $	q(x_{n},x_{m})\ll u,~\forall~m,n\geq N.$
\item [(2)] $\{x_{n}\}$ is said to be $q-$convergent to a point $x\in X$ if for each $u\in E,u\gg 0,$ there exists $N\in \mathbb{Z}^{+}$ so that $q(x,x_{n})\ll u,~\forall~n\geq N.$ By using the concept of lower semicontinuity of $q,$ for each $n\geq N,$ we have $q(x_{n},x)\ll u.$ 
\item [(3)] A cone metric space $(X,d)$ over a topological left $R-$module E is said to be complete with $w-$distance if every $q-$Cauchy sequence is $q-$convergent.  
\end{itemize}
\end{definition}

\begin{theorem}\label{theorem:4}
Let $(X,d)$ be a cone metric space over a topological left $R-$module $E$ and $q$ be a $w-$distance on $X.$ Then every $q-$convergent sequence in $q-$Cauchy.
\end{theorem}

\begin{proof}
Suppose that $\{x_{n}\}$ is a sequence in $X$ convergent to $x.$ Then for all $u\gg 0$ there is a positive integer $N$ such that $q(x_{n},x)\ll u/2$ and $q(x_{m},x)\ll u/2,~\forall~n,m\geq N.$ Now,
\begin{eqnarray*}
q(x_{n},x_{m})&\leq_{P}&q(x_{n},x)+q(x,x_{m})\\
              &\ll&\frac{u}{2}+\frac{u}{2}\\
							&=&u.
\end{eqnarray*}
Hence the proof.
\end{proof}

\begin{theorem}
Let $(X,d)$ be a cone metric space over a topological left $R-$module $E$ and $q$ be a $w-$distance on $X.$ Then every $q-$convergent sequence has a unique limit.
\end{theorem}

\begin{proof}
Let $u\gg 0.$ Now, $q(x,y)\leq_{P}q(x,x_{n})+q(x_{n},y)\ll u+u=2u.$
Then using \ref{proposition:1} (8), we see that $q(x,y)=0.$ 
Proceeding in the similar manner we see that $q(y,x)=0.$ Hence $x=y.$ This proves the theorem.
\end{proof}

\begin{theorem}
Let $(X,d)$ be a cone metric space over a topological left $R$-module $E,q$ be a $w-$distance on $X.$ Suppose $f$ is a function from $X$ into $E$ such that $f(x)\geq_{P} 0_{E},~\forall~x\in X.$ Then a function $p$ from $X\times X$ into $E$ defined by\\
$p(x,y)=f(x)+q(x,y)~\forall~(x,y)\in X\times X$ is a $w-$distance. 
\end{theorem}

\begin{proof}
For each $x,y,z\in X,$ we have  $p(x,z)=f(x)+p(x,z)\leq_{P} f(x)+f(y)+q(x,y)+q(x,z)=p(x,y)+p(y,z).$ It is clear that $f$ is a lower semi-continuous function. Next suppose $\beta \gg 0$ be fixed. As $q$ is a $w-$distance on $X$ there exists $\alpha\gg 0$ so that $q(z,x)\ll \alpha,q(z,y)\ll \alpha$ implies that $d(x,y)\ll \beta.$ Next suppose that $p(z,x)\ll \alpha,p(z,y)\ll \alpha$ we see that $q(z,x)\leq_{P}f(z)+q(z,x)=p(z,x)\ll \alpha,q(z,y)\leq_{P} f(z)+ q(z,y)=p(z,y)\ll \alpha.$ This implies that $d(x,y)\ll \beta.$  
\end{proof}

\begin{theorem}
Let $(X,d)$ be a cone metric space over a topological left $R$-module $E,q$ be a $w-$distance on $X.$ Consider the sequences $\{x_{n}\}$ and $\{y_{n}\}$ in $X.$ Next let $\{a_{n}\}$ and $\{b_{n}\}$ be two sequences in $P$ convergent to $0_{E}$ and $x,y,z\in X.$ Then the following hold:
\begin{itemize}
\item [(i)] If $q(x_{n},y_{n})\leq_{P} a_{n}$ and $q(x_{n},z)\leq_{P} b_{n}$ for all $n\in \mathbb{N},$ then $\{y_{n}\}$ is convergent to $z.$
\item[(ii)] If $q(x_{n},y_{n})\leq_{P} a_{n}$ and $q(x_{n},z)\leq_{P} b_{n}$ for all $n\in \mathbb{N},$ then $\{y_{n}\}$ is a Cauchy sequence.
\item [(iii)] If $q(x_{n},y)\leq_{P}a_{n}$ and $q(x_{n},z)\leq_{P}b_{n}$ for all $n\in \mathbb{N},$ then $y=z.$ 
\item [(iv)] If $q(x_{n},x_{m})\leq_{P} a_{n}$ for all $m,n\in \mathbb{N},$ then $\{x_{n}\}$ is a Cauchy sequence.
\end{itemize} 
\end{theorem}

\begin{proof}
\begin{itemize}
\item [(i)] Let $u\gg 0.$ Then there is $v\gg 0$ such that $q(u^{'},v^{'})\ll v$ and $q(u^{'},z)\ll v$ implies that $d(v^{'},z)\ll u.$ Let $m_{0}$ be a natural number such that $a_{n}\ll v$ and $b_{n}\ll v$ for each $n\geq m_{0}.$ Next for each $n\geq m_{0},q(x_{n},y_{n})\leq_{P}a_{n}\ll v$ and $q(x_{n},z)\leq_{P}b_{n} \ll v.$ Therefore, $d(y_{n},z))\ll u.$ This proves that $y_{n}$ is convergent to $z.$
\item [(ii)] Proof follows from $(i)$ and Theorem \ref{theorem:3}. 
\item [(iii)] Clearly follows from $(i).$
\item [(iv)] Let $u\gg 0.$ Proceeding as proof of $(i)$ let $v\gg 0.$ Then $m_{0}\in \mathbb{N}.$ Next, for $n,m\geq m_{0}+1,q(x_{m_{0}},x_{n})\leq u_{m_{0}}\ll v$ and $q(x_{m_{0}},x_{m})\leq u_{m_{0}}\ll v.$ Therefore, $d(x_{n},x_{m})\ll v.$ This proves that $\{x_{n}\}$ is a Cauchy sequence. 
\end{itemize}
\end{proof}

\begin{theorem}
Let $(X,d)$ be a cone metric space over a topological left $R-$module $E$ with $w-$distance $q$ such that the hypotheses $HypI,HypII$ and $HypIII$ are satisfied. Define $$\mathcal{S}=\{r\in R^{+}|(r^{n}) ~\mbox{is a summable family}\}$$ and $f:X\to X$ be a mapping that satisfies
\begin{eqnarray*}
q(fx,fy)\leq_{P}rq(x,y)~\forall~x,y\in X.
\end{eqnarray*}
Then $f$ has a unique fixed point in $X.$ 
\end{theorem}

\begin{proof}
Let $x_{0}$ be an arbitrary point in $X.$ For each $n\geq 1,$ let $x_{1}=f(x_{0})$ and $x_{n+1}=f(x_{n})=f^{n+1}(x_{0}).$ Then,\\
$q(x_{n},x_{n+1})=q(fx_{n-1},fx_{n})\leq_{P}rq(x_{n-1},x_{n})\leq_{P}r^{2}q(x_{n-2},x_{n-1})\leq_{P}\ldots\leq_{P}r^{n}q(x_{0},x_{1}).$\\
On the basis of above inequality, for each $p\geq 1,$ we have \\
$q(x_{n},x_{n+p})\leq_{P} q(x_{n},x_{n+1})+q(x_{n+1},x_{n+2})+\ldots+q(x_{n+p-1},x_{n+p})\\
\leq_{P}r^{n}q(x_{0},x_{1})+r^{n+1}q(x_{0},x_{1})+\ldots+r^{n+p-1}q(x_{0},x_{1})\\
\leq_{P}(1_{R}+r+\ldots+r^{p-1})r^{n}q(x_{0},x_{1})\\
\leq_{P}(\sum_{i=0}^{+\infty} r^{i})r^{n}q(x_{0},x_{1}).$\\
Using Theorem \ref{theorem:1}, we have $r^{n}\overset {\text{ R}}{\to}0_{R}$ as $n\to \infty.$ Also $\{r^{n}\}$ is a summable family, right multiplication is continuous and using Proposition \ref{proposition:1} (8) we see that for each $u\gg 0,$ there is a natural number $N$ such that $q(x_{n},x_{n+p})\ll u ~\forall~ n\geq N$ and $p\geq 1.$ This proves that $\{x_{n}\}$ is $q-$Cauchy in $X.$ By completeness of $X$ there is $x^{*}$ in $X$ such that $\{x_{n}\}$ is $q-$convergent to $x^{*}$ as $n$ tends to $\infty.$ Take a natural number $N_{2}$ such that $q(x_{n},x^{*})\ll u/2$ and $q(x^{*},x_{n})\ll {u}/{2}$ for each $n\geq N_{2},u\gg0.$ Therefore,
\begin{eqnarray*}
q(fx^{*},x^{*})&\leq_{P}&q(fx^{*},fx_{n})+q(fx_{n},x^{*})\\
              &\leq_{P}&rq(x^{*},x_{n})+q(x_{n+1},x^{*})\\
							&\ll& ru/2+u/2.
\end{eqnarray*}
Now, using Proposition \ref{proposition:1} (1) and (2) we have 
$$q(fx^{*},x^{*})\ll v,~\forall~v>0.$$
Next using Proposition \ref{proposition:1} (7) we see that 
$$q(x^{*},fx^{*})=0.$$
Proceeding in the similar manner we see that $q(f(x^{*},x^{*})=0.$
Therefore, $fx^{*}=x^{*}.$ So, $x^{*}$ is fixed point of $f.$   
Now, if $y^{*}$ is another fixed point of $f.$ Then for each $u\in P^{\circ},$ we have
$$q(x^{*},y^{*})=q(fx^{*},fy^{*})\leq_{P} rq(x^{*},y^{*})\leq_{P} r^{2}q(x^{*},y^{*})\leq_{P}q(x^{*},y^{*})\leq_{P}\ldots\leq_{P}r^{n}\ll u.$$ Thus, $q(x^{*},y^{*})=0.$ Proceeding in the similar manner we see that $q(y^{*},x^{*})=0.$ So, $x^{*}=y^{*}.$ This proves that $f$ has a unique fixed point. 
\end{proof}

\begin{theorem}
Let $(X,d)$ be a cone metric space over a topological left $R-$module $E$ with $w-$distance $q$ such that the hypotheses $HypI,HypII$ and $HypIII$ are satisfied. Define $$\mathcal{S}=\{k\in R^{+}|(k^{n}) ~\mbox{is a summable family}\}$$ and $f:X\to X$ be a mapping that satisfies:
\begin{itemize}
\item [(1)] $q(fx,f^{2}x)\leq_{P}kq(x,fx)~\forall~x\in X,$
\item [(2)] $inf\{q(x,y)+q(x,fx):x\in X\}>_{P}0,$ for each $y\in X$ where $y\neq fy.$
\end{itemize}
 Then there is $\zeta\in X$ so that $\zeta=f\zeta.$ Also, if $\xi=f(\xi),$ then $q(\xi,\xi)=0.$
\end{theorem}

\begin{proof}
Let $v$ be an element of $X.$ Define $v_{n}=f^{n}v,$ where $n\in \mathbb{N}.$ Then we have the following inequality:\\
$q(v_{n},v_{n+1})\leq_{P}kq(v_{n-1},v_{n})\leq_{P}\ldots\leq_{P} k^{n}q(v_{0},v_{1}).$\\
Based on the previous inequality , for all $p\geq 1,$ we have
\begin{eqnarray*}
q(v_{n},v_{n+p})&\leq_{P}&q(v_{n},v_{n+1})+\ldots+q(v_{n+p-1},v_{n+p})\\
                &\leq_{P}&k^{n}q(v_{0},v_{1})+\ldots+k^{n+p-1}q(v_{0},v_{1})\\
								&\leq_{P}&k^{n} (1_{R}+k+\ldots+k^{p-1})q(v_{0},v_{1})\\
								&\leq_{P}&k^{n} (\sum_{i=0}^{+\infty}k^{i})q(v_{0},v_{1}).
\end{eqnarray*}
Using Theorem \ref{theorem:1}, we have $k^{n}\overset {\text{ R}}{\to}0_{R}$ as $n\to \infty.$ Also $\{k^{n}\}$ is a summable family, right multiplication is continuous and using Proposition \ref{proposition:1} (8) we see that for each $u\gg 0,$ there is a natural number $m$ such that $q(v_{n},v_{m})\ll u ~\forall~ n> m.$This proves that $\{v_{n}\}$ is $q-$Cauchy in $X.$ By completeness of $X$ there is $\zeta$ in $X$ such that $\{v_{n}\}$ is $q-$convergent to $\zeta$ as $n$ tends to $\infty.$ Let $n$ be a fixed natural number. As $\{v_{m}\}$ is convergent to $\zeta$ and $q(v_{n},.)$ is lower semi continuous, we see that \\
$q(v_{n},\zeta)\leq_{P} \lim \limits_{m\to \infty} \inf q(v_{n},v_{m})\leq_{P}k^{n}\sum_{i=0}^{+\infty}k^{i}q(v_{0},v_{1}).$\\ Next suppose that $\zeta\neq f\zeta.$ Then we have \\
$0_{E}\leq_{P}\inf\{q(v_{n},\zeta)+q(v_{n},v_{n+1}):n\in \mathbb{N}\}\leq_{P}\inf\{k^{n}\sum_{i=0}^{+\infty} k^{i}q(v_{0},v_{1})+k^{n}q(v_{0},v_{1}):n\in \mathbb{N}\} =0_{E},$\\
a contradiction. Hence $\zeta=f(\zeta).$ If $\zeta=f(\zeta),$ then we have,\\
$q(\zeta,\zeta)=q(f(\zeta),f^{2}(\zeta))\leq k q(\zeta,f(\zeta))=k q(\zeta,\zeta).$\\
Therefore, $q(\zeta,\zeta)\leq k^{n}q(\zeta,\zeta)\ll u,$ for each $u\gg 0.$ Then using Proposition \ref{proposition:1} (7), we see that 
$$q(\zeta,\zeta)=0.$$ Hence the proof.
\end{proof}

\begin{remark}
The statement $``\zeta=f(\zeta)~\mbox{then}~q(\zeta,\zeta)=0"$ gives the application of the previous theorem. In order to get the fixed points of $f,$ we need to find points $\zeta$ where $q(\zeta,\zeta)=0;$ as if $q(\zeta,\zeta)\neq 0,$ then $\zeta\neq f(\zeta).$  
\end{remark}
 
\bibliographystyle{amsplain}

\end{document}